\documentclass[12pt]{amsart}
\usepackage{amssymb,amsmath,amsthm,verbatim}
\usepackage{thmtools,mathtools}
\usepackage[dvipsnames,svgnames,x11names,hyperref]{xcolor}
\usepackage[colorlinks=true,linkcolor=Maroon,citecolor=blue,urlcolor=blue,hypertexnames=false,linktocpage]{hyperref}
\usepackage{lipsum}

\newtheorem{theorem}{Theorem}[section]

\newtheorem{lemma}[theorem]{Lemma}

\newtheorem{corollary}[theorem]{Corollary}
\theoremstyle{definition}
\newtheorem{definition}[theorem]{Definition}

\numberwithin{equation}{section}

\newcommand{\norm}[1]{\left\Vert#1\right\Vert}

\newcommand{\set}[1]{\left\{#1\right\}}
\newcommand{\brac}[1]{\left(#1\right)}
\newcommand{\scalar}[1]{\left \langle #1 \right \rangle}
\newcommand{\sscalar}[1]{\langle #1 \rangle}

\newcommand{\grad}{\operatorname{grad}}
\newcommand{\GL}{\mathrm{GL}}
\newcommand{\II}{\mathrm{II}}
\newcommand{\Leb}{\mathfrak{m}}
\newcommand{\R}{\mathbb{R}}
\renewcommand{\a}{\mathrm{\omega}}
\renewcommand{\S}{\mathbb{S}}
\newcommand{\EE}{\mathcal{E}}
\newcommand{\M}{\mathcal{M}}
\renewcommand{\H}{\mathcal{H}}
\newcommand{\K}{\mathcal{K}}
\newcommand{\F}{\mathcal{F}}

\theoremstyle{remark}
\newtheorem{remark}[theorem]{Remark}

\numberwithin{equation}{section}

\makeatletter
\@namedef{subjclassname@2020}{  \textup{2020} Mathematics Subject Classification}
\makeatother

\begin{document}

\title{$L^p$-Minkowski Problem under Curvature Pinching}
\author[]{Mohammad N. Ivaki\textsuperscript{1} and Emanuel Milman\textsuperscript{2}}

\footnotetext[1]{Institut f\"{u}r Diskrete Mathematik und Geometrie, Technische Universit\"{a}t Wien, Wiedner Hauptstra{\ss}e 8-10, 1040 Wien, Austria. Email: mohammad.ivaki@tuwien.ac.at}

\footnotetext[2]{Department of Mathematics, Technion - Israel
Institute of Technology, Haifa, Israel. Email: emilman@tx.technion.ac.il. \\
This work is supported by the Austrian Science Fund (FWF): Project P36545. The research leading to these results is part of a project that has received funding from the European Research Council (ERC) under the European Union's Horizon 2020 research and innovation programme (grant agreement No 101001677).}

\dedicatory{}
\begin{abstract}
Let $K$ be a smooth, origin-symmetric, strictly convex body in $\R^n$. If for some $\ell\in \GL(n,\R)$, the anisotropic Riemannian metric $\frac{1}{2}D^2 \norm{\cdot}_{\ell K}^2$, encapsulating the curvature of $\ell K$, is comparable to the standard Euclidean metric of $\R^{n}$ up-to a factor of $\gamma > 1$, we show that $K$ satisfies the even $L^p$-Minkowski inequality and uniqueness in the even $L^p$-Minkowski problem for all $p \geq p_\gamma := 1 - \frac{n+1}{\gamma}$. This result is sharp as $\gamma \searrow 1$ (characterizing centered ellipsoids in the limit) and improves upon the classical Minkowski inequality for all $\gamma < \infty$. In particular, whenever $\gamma \leq n+1$, the even log-Minkowski inequality and uniqueness in the even log-Minkowski problem hold. 
\end{abstract}
\maketitle

\section{Introduction}

A central question in contemporary Brunn--Minkowski theory is that of existence and uniqueness in the $L^p$-Minkowski problem for $p \in (-\infty,1)$: given a finite non-negative Borel measure $\mu$ on the Euclidean unit-sphere $\S^{n-1}$, determine conditions on $\mu$ which ensure the existence and/or  uniqueness of a convex body $K$ in $\R^n$ so that
\begin{equation} \label{eq:intro-Lp-Minkowski}
S_p K := h_K^{1-p} S_K  = \mu  . 
\end{equation}
Here $h_K$ and $S_K$ denote the support function and surface-area measure of $K$, respectively -- we refer to Section \ref{sec:prelim} for standard missing definitions. When $h_K \in C^2(\S^{n-1})$, 
\[
S_K = \det(\bar \nabla^2 h_K + h_K \bar \delta) \Leb,
\]
where $\bar \nabla$, $\bar \delta$ and $\Leb$ denote the Levi-Civita connection, standard Riemannian metric and induced Lebesgue measure on $\S^{n-1}$, respectively. Consequently, (\ref{eq:intro-Lp-Minkowski}) is a Monge--Amp\`ere-type equation.

The case $p=1$ above corresponds to the classical Minkowski problem of finding a convex body with prescribed surface-area measure; when $\mu$ is not concentrated on any hemisphere and its barycenter is at the origin, existence and uniqueness (up to translation) of $K$ were established by 
Minkowski, Alexandrov and Fenchel--Jessen (see \cite{Schneider-Book-2ndEd}), and regularity of $K$ was studied by Lewy \cite{Lewy-RegularityInMinkowskiProblem}, Nirenberg \cite{Nirenberg-WeylAndMinkowskiProblems}, Cheng--Yau \cite{ChengYau-RegularityInMinkowskiProblem}, Pogorelov \cite{Pogorelov-MinkowskiProblemBook}, Caffarelli \cite{CaffarelliHigherHolderRegularity,Caffarelli-StrictConvexity} and many others. 
The extension to general $p$ was put forth and promoted by Lutwak \cite{Lutwak-Firey-Sums} as an $L^p$-analog of the Minkowski problem for the $L^p$ surface-area measure $S_p K = h_K^{1-p} S_K$ which he introduced. 
Existence and uniqueness in the class of origin-symmetric convex bodies, when the measure $\mu$ is even and not concentrated in a hemisphere, was established for $n \neq p > 1$ by Lutwak \cite{Lutwak-Firey-Sums} and for $p = n$ by Lutwak--Yang--Zhang \cite{LYZ-LpMinkowskiProblem}. A key tool in the range $p \geq 1$ is the prolific $L^p$-Brunn--Minkowski theory, initiated by Lutwak \cite{Lutwak-Firey-Sums,Lutwak-Firey-Sums-II} following Firey \cite{Firey-Sums},
and developed by Lutwak--Yang--Zhang (e.g. \cite{LYZ-LpAffineIsoperimetricInqs,LYZ-SharpAffineLpSobolevInqs,LYZ-LpJohnEllipsoids}) and others, which extends the classical $p=1$ case. Further existence, uniqueness and regularity results in the range $p > 1$ under various assumptions on $\mu$ were obtained in \cite{ChouWang-LpMinkowski,GuanLin-Unpublished,HuangLu-RegularityInLpMinkowski,HLYZ-DiscreteLpMinkowski,LutwakOliker-RegularityInLpMinkowski,Zhu-ContinuityInLpMinkowski}.

The case $p < 1$ turns out to be more challenging because of the lack of an appropriate $L^p$-Brunn--Minkowski theory. Existence, (non-)uniqueness and regularity under various conditions on $\mu$ were studied by numerous authors when $p<1$, especially after the important work by Chou--Wang \cite{ChouWang-LpMinkowski}, see e.g. \cite{BBC-SmoothnessOfLpMinkowski,BBCY-SubcriticalLpMinkowski,BHZ-DiscreteLogMinkowski,BoroczkyHenk-ConeVolumeMeasure,ChenEtAl-LocalToGlobalForLogBM,ChenLiZhu-LpMongeAmpere,ChenLiZhu-logMinkowski,GLW-SupercriticalLpMinkowski,HeLiWang-MultipleSupercriticalLpMinkowski,  JianLuZhu-UnconditionalCriticalLpMinkowski,LuWang-CriticalAndSupercriticalLpMinkowski,EMilman-Isospectral-HBM,StancuDiscreteLogBMInPlane,Stancu-UniquenessInDiscretePlanarL0Minkowski,Stancu-NecessaryCondInDiscretePlanarL0Minkowski,Zhu-logMinkowskiForPolytopes,Zhu-CentroAffineMinkowskiForPolytopes,Zhu-LpMinkowskiForPolytopes,Zhu-LpMinkowskiForPolytopesAndNegativeP}. The critical exponent for the $L^p$-Minkowski problem is $p=-n$, when the $L^{-n}$ surface-area measure $S_{-n} K$ becomes centro-affine invariant -- see the comments after Corollary \ref{cor:main} below.

Being a self-similar solution to the isotropic Gauss curvature flow, the case $p=0$ and $\mu = \Leb$ of (\ref{eq:intro-Lp-Minkowski}) describes the ultimate fate of a worn stone in a model proposed by Firey \cite{Firey-ShapesOfWornStones}. This model was extended and further studied in 
\cite{Andrews-GaussCurvatureFlowForCurves,Andrews-FateOfWornStones,Andrews-PowerOfGaussCurvatureFlow,AndrewsGuanNi-PowerOfGaussCurvatureFlow,BCD-PowerOfGaussCurvatureFlow,ChoiDaskalopoulos-UniquenessInLpMinkowski,Chow-PowerOfGaussCurvatureFlow,Urbas-PositivePowersOfGaussCurvatureFlow,Urbas-NegativePowersOfGaussCurvatureFlow}.
In the general anisotropic $\alpha$-power-of-Gauss-curvature flow, $x : \S^{n-1} \times [0,T) \rightarrow \R^n$ evolves according to
\begin{equation} \label{eq:intro-anisotropic-flow}
\frac{\partial x}{\partial t} = -  (\rho_K \cdot \kappa_{L_t})^\alpha(\nu_{L_t}(x)) \;  \nu_{L_t}(x) ~,~ \rho_K := \frac{dS_{p} K}{d \Leb} ,
\end{equation}
where $\nu_{L_t}$ is the outer unit-normal to $\partial L_t := x(\S^{n-1},t)$ and $\kappa_{L_t}$ is the corresponding Gauss-curvature. Here $K$ is fixed, and $\rho_K$ serves the role of an anisotropic potential, becoming constant in the isotropic case when $K$ is a Euclidean ball. If $\alpha = \frac{1}{1-p}$, it is straightforward to check that self-similar solutions $L$ to (\ref{eq:intro-anisotropic-flow}) of the form $x(t) = c(t) x$ will satisfy $S_p L = c \, S_p K$. Consequently, to uniquely determine the ultimate fate of a corresponding worn stone, 
one is interested to know whether
\begin{equation} \label{eq:intro-Lp-Minkowski-Uniqueness}
 S_p L = S_p K  \; \Rightarrow \; L = K  . 
\end{equation}
Following contributions in \cite{Andrews-FateOfWornStones,AndrewsGuanNi-PowerOfGaussCurvatureFlow,ChoiDaskalopoulos-UniquenessInLpMinkowski,Chow-PowerOfGaussCurvatureFlow,Firey-ShapesOfWornStones,HuangLiuXu-UniquenessInLpMinkowski}, uniqueness in (\ref{eq:intro-Lp-Minkowski-Uniqueness}) for the general isotropic case in the full range $p \in (-n,1)$ was resolved by Brendle--Choi--Daskalopoulos in \cite{BCD-PowerOfGaussCurvatureFlow} (see also \cite{Saroglou-GeneralizedBCD, IvakiEMilman-GeneralizedBCD} for alternative proofs). In the origin-symmetric case, an extension of their uniqueness result from the case that $K$ is a Euclidean ball to arbitrary centered ellipsoids for any $p \in (-n,1)$ was observed in \cite[Remark 2.7]{EMilman-IsomorphicLogMinkowski}.

 The case $p=0$ is of particular recent importance not only because of its connection to Firey's original problem, but also as it corresponds to the \emph{log-Minkowski problem} for the cone-volume measure
\[
V_K := \frac{1}{n} h_K S_K = \frac{1}{n} S_0 K .
\]
This nomenclature is derived from the fact that $V_K$ is obtained as the push-forward of the cone-measure on $\partial K$ onto $\S^{n-1}$ via the Gauss map, so that the total mass of $V_K$ is $V(K)$, the volume of $K$. 
In \cite{BLYZ-logMinkowskiProblem}, B\"or\"oczky--Lutwak--Yang--Zhang showed that an \emph{even} measure $\mu$ is the cone-volume measure $V_K$ of an \emph{origin-symmetric} convex body $K$ if and only if it satisfies a certain subspace concentration condition, thereby completely resolving the existence part of the \emph{even} log-Minkowski problem. 
As put forth by B\"or\"oczky--Lutwak--Yang--Zhang in their influential work \cite{BLYZ-logBMInPlane,BLYZ-logMinkowskiProblem} and further developed in \cite{KolesnikovEMilman-LocalLpBM,EMilman-IsomorphicLogMinkowski}, the uniqueness question is intimately related to the validity of a conjectured $L^0$- (or log-)Minkowski inequality for a pair of origin-symmetric convex bodies, which would constitute a remarkable strengthening of the classical $p=1$ case.  Note that the restriction to origin-symmetric bodies is natural, and necessitated by the fact that no $L^p$-Minkowski inequality nor uniqueness in the $L^p$-Minkowski problem can hold for general convex bodies when $p < 1$ \cite{Andrews-ClassificationOfLimitingShapesOfIsotropicCurveFlows,ChenLiZhu-LpMongeAmpere,ChenLiZhu-logMinkowski,ChouWang-LpMinkowski,HeLiWang-MultipleSupercriticalLpMinkowski,JianLuWang-NonUniquenessInSubcriticalLpMinkowski,KolesnikovEMilman-LocalLpBM,Li-NonUniquenessInCriticalLpMinkowskiProblem,LLL-NonUniquenessInDualLpMinkowskiProblem,EMilman-Isospectral-HBM,Stancu-UniquenessInDiscretePlanarL0Minkowski}. Establishing the uniqueness in the even log-Minkowski problem (for strictly convex, smooth, origin-symmetric convex bodies) is nowadays a major central open problem of fundamental importance \cite{Boroczky-LogBMSurvey,BLYZ-logMinkowskiProblem,BLYZ-logBMInPlane,HKL-LogBMForSubsets,Kolesnikov-OTOnSphere,KolesnikovLivshyts-ParticularFunctionsInLogBM,KolesnikovEMilman-LocalLpBM,LMNZ-BMforMeasures,EMilman-IsomorphicLogMinkowski,Saroglou-logBM1,Saroglou-logBM2}; see below for additional information and partial results.

In this work, we establish the uniqueness in the even $L^p$-Minkowski problem (\ref{eq:intro-Lp-Minkowski-Uniqueness}) and the corresponding even $L^p$-Minkowski inequality for an appropriate range of $p$'s in $(-n,1)$, with particular emphasis on the case $p=0$, under a curvature pinching condition on $K$, as described next. 

\subsection{Main Results}
A convex compact subset $K$ of $\R^n$ with non-empty interior is called a \emph{convex body}. We shall always assume in this work that the origin is an interior point of $K$. The class of convex bodies in $\R^n$ with $C^{m}$ (respectively, $C^{m,\a}$ for some $\a \in (0,1]$) smooth boundary, strictly positive curvature, and containing the origin in their interior is denoted by $\K^{m}_+$ (respectively, $\K^{m,\a}_+$), $m \geq 2$. The class of origin-symmetric convex bodies in $\R^n$ is denoted by $\K_e$, and we set $\K^{m,\a}_{+,e} := \K^{m,\a}_+ \cap \K_e$.

We denote by $\nu_K(x)$ the outer unit-normal to $K$ at $x \in \partial K$, and by $\II_{\partial K}(x)$ the second fundamental form. The Minkowski gauge function of a convex body $K$ is defined as
\[
\norm{x}_K :=\inf\{t \geq 0: x\in t K\} \quad \forall x\in \R^n ,
\]
and its support function is defined by
\[
h_K(x) := \max_{y\in K} \scalar{x,y} \quad \forall x\in \R^n.
\]
Both functions are convex and positively $1$-homogeneous. 

We denote by $\delta=\sscalar{\cdot\,,\cdot}$ the standard Euclidean metric of $\R^n$, and by $\delta_{\partial K}$ 
the induced Euclidean metric on $\partial K$. The Euclidean unit-ball in $\R^n$ is denoted by $B_2^n$, and the group of non-singular linear transformations is denoted by $\GL(n,\R)$. For a pair of symmetric $2$-tensors $\Theta_1,\Theta_2$ we write $\Theta_1\leq \Theta_2$ to indicate that $\Theta_2 - \Theta_1$ is positive semi-definite. 

\smallskip

To better understand the formulation of our main results below, we first record the following geometric implication of having pinching estimates on $D^2 \frac{\norm{x}_{K}^2}{2}$, which we shall identify as an anisotropic Riemannian metric in Section \ref{sec:prelim}.

\begin{lemma} \label{lem:intro}
Let $K \in \K^2_+$. If $\alpha \delta \leq  D^2 \frac{\norm{x}_{K}^2}{2} \leq \beta \delta$ for some $0 < \alpha \leq \beta$ and all $x \in \partial K$ (equivalently, $x \in \S^{n-1}$), then
\begin{equation} \label{eq:intro-radius-pinching}
\frac{1}{\sqrt{\beta}} B_2^n \subset K \subset \frac{1}{\sqrt{\alpha}} B_2^n
\end{equation}
and
\begin{equation} \label{eq:intro-II-pinching}
 \alpha \delta_{\partial K} \leq \frac{\II_{\partial K}(x)}{\scalar{x,\nu_K(x)}} \leq \beta \delta_{\partial K} \;\;\; \forall x \in \partial K .
 \end{equation}
\end{lemma}

While the above implication is not an equivalence, one can show that up to a degradation of constants, (\ref{eq:intro-radius-pinching}) and (\ref{eq:intro-II-pinching}) imply back a pinching estimate on $D^2 \frac{\norm{x}_{K}^2}{2}$. We will not insist on this here, but rather derive a more precise statement, extending Lemma \ref{lem:intro}, in Lemma \ref{lem:curvature-pinching}.

\begin{theorem}\label{thm:main}
Let $K \in \K^{2,\a}_{+,e}$, $\ell \in \GL(n,\R)$, and  $0 < \alpha < \beta$ be such that
\begin{equation} \label{eq:intro-pinching}
\alpha \delta \leq  D^2 \frac{\norm{x}_{\ell K}^2}{2}\leq \beta \delta \quad \forall x \in \S^{n-1} . 
\end{equation}
Denote
\[
 \gamma := \frac{\beta}{\alpha} \in (1,\infty), \quad\quad  p_\gamma := 1 - \frac{n+1}{\gamma} .
\]
Then for all $p \in [p_\gamma,1)$, the even $L^p$-Minkowski problem for $K$ has a unique solution:
\begin{equation} \label{eq:main-Lp-Minkowski-Uniqueness}
\forall L \in \K_e, \;\; \; \;  S_p L = S_p K  \; \Rightarrow \; L = K ,
\end{equation}
and $K$ satisfies the even $L^p$-Minkowski inequality:
\begin{equation} \label{eq:main-Lp-Minkowski-Inq}
\forall L \in \K_e  \;\;\;\; \frac{1}{p} \int_{\S^{n-1}} h_L^{p} dS_p K  \geq \frac{n}{p} V(K)^{1-\frac{p}{n}} V(L)^{\frac{p}{n}} ,
\end{equation}
with equality if and only if $L = c K$ for some $c > 0$.
\end{theorem}

Applying this to the case $\gamma = n+1$, we immediately obtain a condition ensuring uniqueness in the important logarithmic case $p=0$, which is interpreted in the limiting sense as follows: 
\begin{corollary} \label{cor:main}
With the same assumptions and notation as in Theorem \ref{thm:main}, assume that
\[
\beta \leq (n+1) \alpha . 
\]
Then the even log-Minkowski problem for $K$ has a unique solution:
\begin{equation} \label{eq:main-log-Minkowski-Uniqueness}
\forall L \in \K_e, \;\; \; \;  V_L = V_K  \; \Rightarrow \; L = K ,
\end{equation}
 and $K$ satisfies the even log-Minkowski inequality:
\begin{equation} \label{eq:main-log-Minkowski-Inq}
 \forall L \in \K_e \;\;\; \frac{1}{V(K)} \int_{\S^{n-1}} \log \frac{h_L}{h_K} dV_K \geq  \frac{1}{n} \log \frac{V(L)}{V(K)} ,
 \end{equation}
with equality if and only if $L = c K$ for some $c > 0$. 
\end{corollary}

In fact, we shall see that our proof yields a slightly better estimate for $p_\gamma$ as a function of $\gamma$, but we have stated the most elegant formulation above. Note that $\gamma < \infty$ by compactness of $\S^{n-1}$ and strict convexity of $K$, and so the above results always yield a non-trivial improvement over the classical uniqueness in the ($L^1$-) Minkowski problem and Minkowski inequality, which however apply to arbitrary (possibly non-symmetric nor smooth) convex bodies. As explained in \cite[Remark 1.4]{EMilman-IsomorphicLogMinkowski}, the inequalities (\ref{eq:main-Lp-Minkowski-Inq}) and (\ref{eq:main-log-Minkowski-Inq}) also hold for $K \in \K_e$ which are obtained as Hausdorff limits of $K_i$'s satisfying the assumption (\ref{eq:intro-pinching}).

Furthermore, note that Theorem \ref{thm:main} is sharp for the case that $\gamma= \frac{\beta}{\alpha} = 1$, when $K$ is necessarily an origin-symmetric ellipsoid $\EE$. That case was excluded from the formulation since $p_1 = 1 - \frac{n+1}{1} = -n$ is the critical exponent for the $L^p$-Minkowski problem, due to the $\GL(n,\R)$ equivariance of the centro-affine Gauss curvature $\frac{\Leb}{S_{-n} K} = \kappa_K / h_K^{1+n}$ \cite{Tzitzeica1908,ChouWang-LpMinkowski}; in particular, the centro-affine Gauss curvature of any centered ellipsoid $\EE$ in $\R^n$ is constant and depends only on its volume: $S_{-n} \EE = c_n V(\EE)^2 \Leb$, so no uniqueness can hold when $p=-n$. However, applying Theorem \ref{thm:main} with $\gamma \searrow 1$, we recover  
(in the origin-symmetric case) the sharp uniqueness result by Brendle--Choi--Daskalopoulos \cite{BCD-PowerOfGaussCurvatureFlow} in the full range $p \in (-n,1)$ and its extension to centered ellipsoids from \cite[Remark 2.7]{EMilman-IsomorphicLogMinkowski}. We remark that centered ellipsoids are in fact the only complete elliptic hypersurfaces with constant centro-affine Gauss curvature, as shown by Calabi \cite{Calabi-CompleteAffineHyperspheres}.

\subsection{Comparison with previous work}

The validity of the log-Minkowski inequality (\ref{eq:main-log-Minkowski-Inq}) for all $K \in \K_e$, including characterization of its equality cases, as well as the uniqueness in the even log-Minkowski problem (\ref{eq:main-log-Minkowski-Uniqueness}) for $K \in \K_e$ which is not a parallelogram, was established when $n=2$ by B\"or\"oczky--Lutwak--Yang--Zhang \cite{BLYZ-logBMInPlane} (see also \cite{MaLogBMInPlane,Putterman-LocalToGlobalForLpBM,XiLeng-DarAndLogBMInPlane,Xi-ReverseLogBM} for alternative derivations). The even log-Minkowski inequality has also been established in $\R^n$ for all zonoids $K$ \cite{VanHandel-LogMinkowskiForZonoids, Xi-ReverseLogBM}, under various symmetry assumptions \cite{BoroczkyKalantz-LogBMWithSymmetries,Rotem-logBM,Saroglou-logBM1}, for perturbations of the Euclidean ball \cite{ChenEtAl-LocalToGlobalForLogBM} (following \cite{ColesantiLivshyts-LocalpBMUniquenessForBall,CLM-LogBMForBall,KolesnikovEMilman-LocalLpBM}), and locally for perturbations of linear images of the unit-balls of $\ell_p^n$ \cite{KolesnikovEMilman-LocalLpBM}. 
However, the general case remains open for $n \geq 3$.

By combining the results of \cite{KolesnikovEMilman-LocalLpBM,ChenEtAl-LocalToGlobalForLogBM, Klartag-AlmostKLS}, the even $L^p$-Minkowski inequality (\ref{eq:main-Lp-Minkowski-Inq}) and uniqueness in the even $L^p$-Minkowski problem (\ref{eq:main-Lp-Minkowski-Uniqueness}) are known to hold in $\R^n$ for $p \in [1 - \frac{c}{n \log(1+n)} , 1)$, for some universal constant $c > 0$. Consequently, Theorem \ref{thm:main} gives no new information whenever $\gamma \geq C n^2 \log (1+n)$.

The conclusion of Theorem \ref{thm:main} was previously established for a worse range of $p$'s and under a slightly weaker curvature pinching estimate in \cite[Theorem 1.2, Theorem 6.3]{EMilman-IsomorphicLogMinkowski}. Namely, it was shown that if (\ref{eq:intro-II-pinching}) holds for $\ell K$ and in addition $\ell K \subset R B_2^n$, then the conclusion of Theorem \ref{thm:main} holds for all $p \in (p_{\alpha,\beta,R} , 1)$ with
\[
p_{\alpha,\beta,R} := 2 - \frac{n-1}{2} \frac{\alpha}{\beta} + R^2 \alpha . 
\]
It follows that if $\gamma = \frac{\beta}{\alpha}$ is too large, $p_{\alpha,\beta,R} > 1$ and no information is obtained. Moreover, when $\gamma=1$ (hence $\ell K$ is necessarily a Euclidean ball and $R^2 \alpha = 1$), we have $p_{\alpha,\beta,R} = 3 - \frac{n-1}{2}$, and so the sharp exponent $p=-n$ is not obtained. Both of these drawbacks are absent from the estimate $p_{\gamma} = 1 - \frac{n+1}{\gamma}$ obtained in Theorem \ref{thm:main}.

The proof of Theorem \ref{thm:main} involves several improvements over the argument of \cite{EMilman-IsomorphicLogMinkowski}. As in that work, a key role is played by the centro-affine differential calculus.

\medskip
\textbf{Acknowledgment.} We thank Yingxiang Hu and the referees for their helpful comments.

\section{Preliminaries} \label{sec:prelim}

We work in Euclidean space $\R^n$, which we also identify with its dual using the Euclidean metric $\delta = \scalar{\cdot\, ,\cdot}$. 
Let $D$ denote the flat connection on $\R^n$. 
Let  $(\S^{n-1},\bar{\delta},\bar{\nabla})$ denote the unit sphere equipped with its standard round metric and Levi-Civita connection.
We denote by $\Leb$ the spherical Lebesgue measure on $\S^{n-1}$.

\subsection{Convex Geometry}

Let $K \in \K^2_+$ be a $C^2$-smooth convex body with strictly positive curvature, containing the origin in its interior. Recall that $\II_{\partial K}$ denotes the second fundamental form of $\partial K$.
The Gauss map $\nu_K: \partial K \to \S^{n-1}$ sends $x \in\partial K$ to its unique outer unit normal vector.
The support function of $K$ is defined as $h_K(x) = \max_{y \in K} \scalar{x,y} : \R^n \rightarrow \R_+$, and we have $\scalar{\nu_K(x) , x} = h_K(\nu_K(x))$ for $x \in \partial K$. In addition, we have
\[
X_K(\theta) := \nu_K^{-1}(\theta) = D h_K(\theta)  : \S^{n-1} \rightarrow \partial K ,
\]
and therefore (invoking the $1$-homogeneity of $h_K$)
\begin{equation} \label{eq:II-D2hK}
(\nu_K)_* \mathrm{II}_{\partial K}  = \bar D^2 h_K := D^2 h_K|_{T \S^{n-1} \times T \S^{n-1}} = \bar \nabla^2 h_K + h_K \bar{\delta} .
\end{equation}

 The surface-area measure $S_K$ is the measure on $\S^{n-1}$ given by
 \[
 S_K := (\nu_K)_*(\H^{n-1}|_{\partial K}) = \frac{1}{\kappa_K} \Leb ,
\]
where $\kappa_K(\theta)$ denotes the Gauss curvature of $\partial K$ at $X_K(\theta)$:
\[
\frac{1}{\kappa_K(\theta)}=\frac{\det (\bar{\nabla}^2 h_K + h_K \bar{\delta})}{\det(\bar{\delta})}\Big|_\theta, \quad\quad \theta \in \S^{n-1}.
\]
More generally, the $L^p$ surface-area measure of $K$ is defined as
\[
 S_p K := h^{1-p}_K S_K . 
\]
The cone-volume measure $V_K$ on $\S^{n-1}$ is defined as
\begin{equation} \label{eq:VK}
V_K = \frac{1}{n} h_K S_K = \frac{1}{n} S_0 K ;
\end{equation}
it is obtained by first pushing forward the Lebesgue measure on $K$ via the composition of the cone-map $K \ni x \mapsto x / \norm{x}_K \in \partial K$ with the Gauss map $\nu_K : \partial K \rightarrow \S^{n-1}$.

The associated Minkowski gauge function  $\norm{\cdot}_K : \R^n \rightarrow \R_+$ is defined as the positively $1$-homogeneous function $\inf \{ t \geq 0 : x \in t K\}$, or equivalently, as the support function $h_{K^\ast}$ of the polar body defined by
\[
K^{\ast}:=\{x\in \R^n:  \scalar{x,y} \leq 1,\,\forall y\in K\}.
\]
Convexity of $K$ implies that $\norm{\cdot}_K$ and $h_K$ are convex; when $K$ is in addition origin-symmetric, these define norms on $\R^n$ whose unit-balls are $K$ and $K^{\ast}$, respectively. We refer to \cite{Schneider-Book-2ndEd} for additional background in Convex Geometry. 

\subsection{Anisotropic metric}

\begin{definition}[Anisotropic Riemannian metric]
The anisotropic Riemannian metric induced by $K$ on $\R^n \setminus \{0\}$ is defined as the positive semi-definite symmetric $2$-tensor given by
\begin{equation} \label{eq:PK}
	P_K(x) := D^2 \frac{\norm{x}_K^2}{2} , \quad x\in \R^n\setminus\{0\}.
\end{equation}
\end{definition}

The strictly positive curvature of $K$ implies that $P_K$ is in fact positive-definite. Since $x \mapsto P_K(x)$ is $0$-homogeneous, it is enough to verify this on $\partial K$. Along the way, we observe that the radial and tangential directions to $\partial K$ are $P_K$-orthogonal: 

\begin{lemma}\label{lem:PK}
Let $x\in \partial K$. Then $P_K(x)(x,x)=1$, and for all $u,v \in T_x \partial K$,
\begin{equation} \label{eq:barPK}
P_K(x)(x,u) = 0, \quad P_K(x)(u,v) = \frac{\mathrm{II}_{\partial K}(u,v)}{h_K \circ \nu_K(x)}  . 
\end{equation}
In particular, $\bar P_K := P_K|_{T \partial K \times T \partial K}$ is a metric on $\partial K$ and its push-forward onto $\S^{n-1}$ via $\nu_K$ is given by
\begin{equation} \label{eq:gK}
g_K := (\nu_K)_*(\bar P_K) = \frac{\bar D^2 h_K}{h_K} . \end{equation}
\end{lemma}

\begin{proof}
Since $D \frac{\norm{\cdot}_K^2}{2}$ is homogeneous of degree one, we have by Euler's identity:
\begin{align} \label{eq:Px}
P_K(x)(x,\cdot)=  D \frac{\norm{x}^2_K}{2} .
\end{align}
Since $\partial K = \{x\in \R^n: \norm{x}_K = 1\}$, it follows that $P_K(x)(x,u) = 0$, and after another application of Euler's identity, that $P_K(x)(x,x)= \norm{x}_K^2 = 1$. In addition, we observe that for all $u,v \in T_x \partial K$:
\begin{equation} \label{eq:Q}
D^3 \frac{\norm{x}_K^2}{2} ( x, u , v) = 0 . 
\end{equation}

Since for $x \in \partial K$ we have $\scalar{D \norm{x}_K , x} = \norm{x}_K = 1$, $\nu_K(x)$ is parallel to $D \norm{x}_K$ and $\scalar{\nu_K(x),x} = h_K(\nu_K(x))$, we conclude that
\[
D \norm{x}_K=\frac{\nu_K(x)}{h_K(\nu_K(x))}\quad \forall x\in\partial K.
\]
Using local coordinates on $\partial K$, it follows from (\ref{eq:Px}) that
\begin{align}\label{eq:anisotropic-form}
	P_K(x)(x,\partial^2_{ij}x)=\frac{\sscalar{\nu_K(x), \partial_{ij}^2 x}}{h_K(\nu_K(x))}
=-\frac{\mathrm{II}_{\partial K}(\partial_i x, \partial_j x)}{h_K\circ \nu_K}\Big|_{x} .
\end{align}
Recalling that $P_K(x)(x,\partial_i x) = 0$, we apply $\partial_j$. Using (\ref{eq:Q}) and (\ref{eq:anisotropic-form}), it follows that
\[
P_K(x)(\partial_i x, \partial_j x) = \frac{\mathrm{II}_{\partial K}(\partial_i x, \partial_j x)}{h_K\circ \nu_K}\Big|_{x} .
\]
The strictly positive curvature of $K$ implies that $\bar P_K$ is a positive-definite metric on $\partial K$. The last assertion follows by (\ref{eq:II-D2hK}). 
\end{proof}

It may be elucidating to explicitly express $P_K$ as follows:
\begin{lemma} \label{lem:curvature-pinching}
Given $K \in \K^2_+$, let $\pi_{\partial K}(x) : T_x \R^n \rightarrow T_x \partial K$ denote the projection onto $T_x \partial K$ parallel to the radial direction $x$. Then for all $x \in \partial K$:
\begin{equation} \label{eq:detailed-pinching}
D^2 \frac{\norm{x}_{K}^2}{2} = \pi_{\partial K}(x)^T \frac{\II_{\partial K}(x)}{\scalar{x,\nu_K(x)}} \pi_{\partial K}(x) + (\textrm{Id} - \pi_{\partial K}(x))^T \frac{1}{|x|^2} (\textrm{Id} - \pi_{\partial K}(x)) . 
\end{equation}
\end{lemma}
\begin{proof}
Since the radial and tangential directions are $P_K$-orthogonal, namely $P_K(x)(x,u) = 0$ for all $u \in T_x \partial K$, it follows that:
\[
P_K(x) = \pi_{\partial K}(x)^T P_K(x) \pi_{\partial K}(x) +  (\textrm{Id} - \pi_{\partial K}(x))^T P_K(x) (\textrm{Id} - \pi_{\partial K}(x)) . 
\]
As $\textrm{Id} - \pi_{\partial K}(x)$ is the projection onto the radial direction spanned by $x$ (in parallel to $T_x \partial K$), and since $P_K(x)(t x/|x|, t x/|x|) = \frac{t^2}{|x|^2}$, we obtain the second summand in (\ref{eq:detailed-pinching}). The first summand is immediately deduced from (\ref{eq:barPK}), after recalling that $h_K(\nu_K(x)) = \scalar{x,\nu_K(x)}$. 
\end{proof}

Conjugating both sides of (\ref{eq:detailed-pinching}) with $\pi_{\partial K}(x)$ and $\textrm{Id} - \pi_{\partial K}(x)$, and noting that $\pi_{\partial K}(x) (\textrm{Id} - \pi_{\partial K}(x)) = (\textrm{Id} - \pi_{\partial K}(x))  \pi_{\partial K}(x)  = 0$, Lemma \ref{lem:intro} from the Introduction immediately follows. Recall that $x \mapsto P_K(x)$ is $0$-homogeneous, and so controlling $P_K$ on $\partial K$, $\S^{n-1}$ or $\R^n \setminus \{0\}$ are all equivalent.
We refer to \cite{Xia-AnisotropicMinkowski,WeiXiong-AnisotropicCurvatureFlow} for additional background on anisotropic geometry. 

\subsection{Centro-affine differential geometry}

Let $X: \S^{n-1}\to \M$ be a smooth embedding of $\S^{n-1}$ into $\R^n$ so that the centro-affine normal field $\xi(\theta)=X(\theta)$ is transversal to $\M$. 
In particular, this holds when $X = X_K = D h_K$ and $\M = \partial K$, where $K \in \K^\infty_+$ is a $C^\infty$-smooth convex body with strictly positive curvature, containing the origin in its interior. We identify between the tangent spaces $T_x \R^n$ and $\R^n$, but for the sake of clarity preserve in this preliminary subsection the distinction between the centro-affine normal vector $\xi$ and the embedding $X$. 

The centro-affine normal $\xi$ induces a volume form $V$, a connection $\nabla$, as well as a bilinear form $g^{\xi}$ on $\S^{n-1}$ as follows:
\begin{align*}
V(e_1,\ldots, e_{n-1}):= \frac{1}{n} \det (dX(e_1),\ldots, dX(e_{n-1}),\xi),\quad e_i\in T\S^{n-1},	
\end{align*}
\begin{align}\label{centro-affine}
D_udX(v)=dX(\nabla_u v)-g^{\xi}(u,v)\xi, 
\end{align}
for all smooth vector fields $u,v$ on $\S^{n-1}$. 
Moreover, $g^{\xi}$ is symmetric and positive-definite, $\nabla$ is torsion-free, 
and $\nabla V \equiv 0$. 
The conormal field $\xi^\ast: \S^{n-1}\to (\R^n)^{\ast} \simeq \R^n$ is defined as the unique smooth vector field on the dual space of $\R^n$ such that $\scalar{ \xi^\ast,dX}=0$ and $\scalar{ \xi, \xi^\ast}=1$. 
Since $\xi^\ast$ is an immersion and transversal to its image, it also induces a symmetric bilinear form $g^{\xi^{\ast}}$ and a torsion-free connection $\nabla^{\ast}$ on $\S^{n-1}$ via
\[
D_ud\xi^\ast(v)=d\xi^{\ast}(\nabla_u^{\ast}v)-g^{\xi^\ast}(u,v)\xi^\ast.
\]
It can be shown that
\[
g(u,v) := g^{\xi}(u,v) = \scalar{ d\xi^\ast(u), d\xi(v)} =g^{\xi^{\ast}}(u,v).
\]
The connection $\nabla^{\ast}$ is the conjugate connection to $\nabla$ with respect to $g$; that is, for vectors fields $u,v_1,v_2$, we have
\[ u (g(v_1,v_2))=g(\nabla_uv_1,v_2)+g(v_1,\nabla^{\ast}_uv_2).
\] Equivalently, if $\omega$ is a $1$-form, this means in a local frame that
\begin{equation} \label{eq:conjugate}
\nabla_i (g^{jk} \omega_k) = g^{jk} \nabla^*_i \omega_k .
\end{equation}
We furnish all geometric quantities associated with $\xi^{\ast}$ with $\ast.$ The triplet $(X,\xi,\xi^{\ast})$ is called a centro-affine hypersurface, and the associated geometry is invariant under the centro-affine group $\GL(n,\mathbb{R})$. We refer to \cite{NomizuSasaki-Book,LSZH-Book,EMilman-IsomorphicLogMinkowski} for additional information on centro-affine differential geometry.

\begin{lemma}\label{induced centro-affine geometry}
Let $X: \S^{n-1}\to \partial K$ be given by $X=X_K = D h_K$. Then $\xi^{\ast} : \S^{n-1}  \rightarrow \partial K^{\ast}$ is given by $\xi^\ast(\theta)= \frac{\theta}{h_K(\theta)}$ and
\[
	g = g_K,\quad V  = V_K ,
\]
where recall $g_K$ and $V_K$ were defined in (\ref{eq:gK}) and (\ref{eq:VK}), respectively.
\end{lemma}
\begin{proof}
The claims follow from \eqref{centro-affine} and taking the $\delta$ inner product of both sides of \eqref{centro-affine} with $\nu_K$; see \cite[Proposition 4.2]{EMilman-IsomorphicLogMinkowski}.
\end{proof}

\begin{remark}\label{key rem}
The metric $g_K$ constructed in Lemma \ref{lem:PK} from the anisotropic Riemannian metric $P_K$ is thus seen to be the centro-affine metric induced by $\partial K$, a fact which we shall freely use. It turns out that both primal and dual structures  $(\nabla_K,g_K)$ and $(\nabla^{\ast}_K,g_K)$ have constant sectional and Ricci curvatures, turning any $(\S^{n-1},\nabla_K,g_K)$ into a centro-affine unit-sphere; cf. \cite{LSZH-Book,EMilman-IsomorphicLogMinkowski}. 
\end{remark}

We henceforth fix $K$ and use $\nabla_K$ and $\nabla_K^{\ast}$ to denote the induced primal and conjugate centro-affine connections. We abbreviate $\grad=\grad_{g_K}$, and use $|\cdot|_{g_K}$ to denote the norm with respect to the centro-affine metric $g_K$.

\subsection{Centro-affine differential calculus}

For a function $f \in C^2(\S^{n-1})$, the Hessian with respect to the conjugate centro-affine structure $(g_K,\nabla^{\ast}_K)$ is the following symmetric $2$-tensor:
\begin{align}\label{hess ast def}
\operatorname{Hess}^{\ast}_K f(u,v) = \nabla^{\ast}_K df(u,v)=v(uf)-df( (\nabla_K^{\ast})_v u) .
\end{align}
The centro-affine Laplacian with respect to $(g_K,\nabla_K)$ (also called the Hilbert--Brunn--Minkowski operator) is defined as (recall (\ref{eq:conjugate}))
\begin{equation} \label{eq:Laplacian}
\Delta_K f = \nabla_K \cdot \grad_{g_K} f = \operatorname{tr}_{g_K}\operatorname{Hess}^{\ast}_K f .
\end{equation}
In local coordinates \cite[Section 2.3]{EMilman-IsomorphicLogMinkowski}: \[
\Delta_K f=((\bar D^2 h_K)^{-1})^{ij}[\bar{\nabla}_{i,j}^2(f h_K)+(f h_K)\bar{\delta}_{ij}]- (n-1) f.
\]
Integrating by parts using $\nabla_K V_K \equiv 0$, one has for all $f_1,f_2 \in C^2(\S^{n-1})$:
\[
-\int f_1 \Delta_K f_2 \, dV_K = \int \scalar{\grad f_1 , \grad f_2}_{g_K} dV_K = - \int (\Delta_K f_1) f_2 \, dV_K .
\]
Consequently, $-\Delta_K$ is self-adjoint (with dense domain $H^2(\S^{n-1})$) and positive semi-definite on $L^2(V_K)$. Its spectrum $\sigma(-\Delta_K)$ is discrete, consisting of a countable sequence of eigenvalues of finite multiplicity starting at $0$ (corresponding to constant functions) and increasing to $\infty$. The spectrum is invariant under centro-affine transformations: $\sigma(-\Delta_{\ell K}) = \sigma(-\Delta_{K})$ for all $\ell \in \GL(n,\R)$. We refer to \cite[Section 5]{KolesnikovEMilman-LocalLpBM} and \cite[Sections 3,4]{EMilman-IsomorphicLogMinkowski} for further details.

Using that the centro-affine Ricci curvature is constant $(n-2) g_K$, the following centro-affine Bochner formula was established in 
\cite[Theorem 5.2]{EMilman-IsomorphicLogMinkowski}:
\begin{align} \label{eq:Bochner}
\int (\Delta_K f)^2 dV_K = \int \brac{|\operatorname{Hess}^{\ast}_K f|_{g_K}^2 + (n-2)|\grad f|_{g_K}^2} dV_K .
\end{align}

\subsection{Local Brunn-Minkowski inequality}

The local Brunn--Minkowski inequality, a Poincar\'e-type inequality going back to Hilbert following Minkowski \cite[pp. 108--109]{BonnesenFenchelBook} and rediscovered by Colesanti \cite{ColesantiPoincareInequality}, is an infinitesimal form of the classical Brunn-Minkowski inequality \cite{Schneider-Book-2ndEd} -- see \cite[Section 5]{KolesnikovEMilman-LocalLpBM} for an exposition. In the centro-affine language, it states that
\begin{align}
\label{L1BM} \int f dV_K = 0 \; \Rightarrow \; \int (-\Delta_K f) f dV_K = \int |\grad f|_{g_K}^2 dV_K \geq (n-1) \int f^2 dV_K ,
\end{align}
with equality if and only if for some $v \in \R^n$ we have \begin{equation} \label{eq:eigenspace}
f(\theta)=\scalar{\theta/ h_K(\theta),v} \quad \forall \theta \in \S^{n-1}
\end{equation}
(see \cite[Theorems 5.3 and 5.4]{EMilman-IsomorphicLogMinkowski} for a simple proof using (\ref{eq:Bochner})). Defining $\lambda_1(-\Delta_K)$ to be the first non-zero eigenvalue of $-\Delta_K$, it follows that $\lambda_1(-\Delta_K) = n-1$ (with multiplicity $n$).

When $K$ is in addition assumed to be origin-symmetric, denote by $C^2_e(\S^{n-1})$ and  $H^2_e(\S^{n-1})$ the even elements of $C^2(\S^{n-1})$ and $H^2(\S^{n-1})$, respectively, and by $\mathbf{1}^{\perp}$ those elements $f \in H^2(\S^{n-1})$ for which $\int f  \, dV_K = 0$. The first non-zero \emph{even} eigenvalue of $-\Delta_K$ is defined as
\begin{align*}
 \lambda_{1,e}(-\Delta_K) & := \min \sigma(-\Delta_K|_{H^2_e(\S^{n-1}) \cap \mathbf{1}^{\perp}}) \\
 & = \inf \set{ \frac{\int |\grad f|_{g_K}^2 dV_K}{\int f^2 dV_K} \; : \;  0 \neq f \in C^2_{e}(\S^{n-1}) , \int_{\S^{n-1}} f \, dV_K = 0 } .
\end{align*}
Since the eigenspace (\ref{eq:eigenspace}) corresponding to $\lambda_1(-\Delta_K) = n-1$ consists of odd functions, it follows that $\lambda_{1,e}(-\Delta_K) >  n-1$. 

\subsection{Local-to-global principle}

The following local-to-global principle was explicitly formulated in \cite{EMilman-IsomorphicLogMinkowski}, after combining several key ingredients from \cite{KolesnikovEMilman-LocalLpBM,ChenEtAl-LocalToGlobalForLogBM, BCD-PowerOfGaussCurvatureFlow}; in fact, the formulation below is slightly stronger than the one appearing in \cite[Theorem 2.1]{EMilman-IsomorphicLogMinkowski}, but readily follows from the proof given there. 
The $C^{2,\a}$-metric on $\K^{2,\a}_{+,e}$ is given by the $C^{2,\a}$-norm of the difference of support functions on $\S^{n-1}$. 

\begin{theorem} \label{thm:l2g}
Let $\F \subset \K^{2,\a}_{+,e}$ be any subfamily containing $B_2^n$ which is path-connected in the $C^{2,\a}$ topology. Namely, for any $K \in \F$, there exists $[0,1] \ni t \mapsto K_{t} \in \F$ a continuous path in the $C^{2,\a}$ topology so that $K_0 = B_2^n$ and $K_1 = K$. Then given $p \in (-n,1)$, each of the following statements implies the next:
\begin{enumerate}
\item 
For all $K \in \F$, $\lambda_{1,e}(-\Delta_K) > n-p$ (with strict inequality!). 
\item 
For all $K \in \F$, uniqueness holds in the even $L^p$-Minkowski problem (\ref{eq:main-Lp-Minkowski-Uniqueness}).
\item 
For all $K \in \F$, the even $L^p$-Minkowski inequality (\ref{eq:main-Lp-Minkowski-Inq}) holds (with the case $p=0$ interpreted in the limiting sense of (\ref{eq:main-log-Minkowski-Inq})),
with equality if and only if $L = c K$ for some $c > 0$.
\end{enumerate}
\end{theorem}

Given $\gamma \geq 1$, let $\F_\gamma$ denote the collection of convex bodies $K  \in \K^{2,\a}_{+,e}$ so that there exists $\ell \in \GL(n,\R)$ and $0 < \alpha \leq \beta$ with $\beta/\alpha = \gamma$ so that the curvature-pinching bounds (\ref{eq:intro-pinching}) hold. Clearly $B_2^n \in \F_\gamma$. In view of Theorem \ref{thm:l2g}, in order to establish Theorem \ref{thm:main}, it is enough to establish that $\F_\gamma$ is path-connected in the $C^{2,\a}$ topology and to show that the following local spectral estimate holds:
\begin{equation} \label{eq:local}
\gamma > 1 \;\; \Rightarrow \;\; \forall K \in \F_\gamma , \;\;\; \lambda_{1,e}(-\Delta_K) > n - p_\gamma  .
\end{equation} 
The heart of the proof will be to establish (\ref{eq:local}) in the next section. As for the $C^{2,\a}$-path-connectedness of $F_\gamma$, this follows immediately by linearly interpolating $\norm{\cdot}_{K}^2$ and $\norm{\cdot}_{B_2^n}^2$ (after assuming without loss of generality that $\ell = \textrm{Id}$ and $\alpha \leq 1 \leq \beta$; see the proof of \cite[Theorem 6.4]{EMilman-IsomorphicLogMinkowski}).

\section{Local spectral estimate}

In view of the previous comments, in order to conclude the proof of Theorem \ref{thm:main}, it remains to establish the local spectral estimate (\ref{eq:local}). 

\begin{theorem} \label{thm:spectral}
With the same assumptions and notation as in Theorem \ref{thm:main},
\[
\gamma > 1 \;\; \Rightarrow \;\; \lambda_{1,e}(-\Delta_K) > n - p_\gamma .
\]
Moreover, if (\ref{eq:intro-pinching}) holds with $\alpha=\beta$ so that $\gamma=1$, then $K$ is a centered ellipsoid and equality holds: 
\[
\gamma =1 \;\; \Rightarrow \;\; \lambda_{1,e}(-\Delta_K) = n - p_1 = 2n . 
\]
\end{theorem}

\begin{remark} \label{rem:upper-bound}
It was shown in \cite[Theorem 1.1]{EMilman-Isospectral-HBM} that for all $K \in \K^2_{+,e}$ we have  $\lambda_{1,e}(-\Delta_K) \leq 2n$, with equality if and only if $K$ is a centered ellipsoid.
\end{remark}

\begin{proof}[Proof of Theorem \ref{thm:spectral}]
First, note that it is enough by approximation to establish the spectral estimate for $K \in \F_{\gamma} \cap \K^{\infty}_{+,e}$, since it is known \cite[Theorem 5.3]{KolesnikovEMilman-LocalLpBM} that the eigenvalues of $-\Delta_{K}$ are continuous in $K$ with respect to the $C^2$ (and in particular, $C^{2,\a}$) topology, the curvature pinching bounds $\alpha,\beta$ also vary continuously in the latter topology, and $p_\gamma$ is continuous in $\alpha,\beta$. Furthermore, since the spectrum of $-\Delta_K$ is centro-affine invariant, we may assume without loss of generality that $\ell = \mathrm{Id}$.

Consequently, let $K$ be an origin-symmetric smooth convex body in $\R^n$ with strictly positive curvature so that $\alpha \delta \leq P_K(x) \leq \beta \delta$ for all $x \in \R^n \setminus \{0\}$, where, recall, $P_K$ is the anisotropic Riemannian metric defined in (\ref{eq:PK}). Further recall that $X_K := D h_K : \S^{n-1} \rightarrow \partial K$ is the inverse Gauss map, which also plays the role of the centro-affine normal after identifying between $\R^n$ and $T_{X_K(\theta)} \R^n$. For any $\theta \in \S^{n-1}$, $p_K(\theta):=P_K(X_K(\theta))$ defines a Euclidean norm on $T_{X_K(\theta)} \R^n$, and for all $\theta \in \S^{n-1}$ and $v \in T_{X_K(\theta)} \R^n$: 
\begin{equation} \label{eq:ab}
\alpha|v|_{\delta}^2\leq |v|_{p_K(\theta)}^2\leq \beta|v|_{\delta}^2.
\end{equation}

Let $f \in C^2_e(\S^{n-1})$. Given a constant $\tau \in \R$ to be determined later, define
\begin{align*}
F&=dX_K(\grad_{g_K} f)+\tau f X_K : \S^{n-1} \rightarrow T_{\partial K} \R^n ,
\end{align*}
where we use $T_{\partial K} \R^n$ to denote the subbundle of $T \R^n$ over $\partial K$ (so for each $\theta \in \S^{n-1}$, $F(\theta) \in T_{X_K(\theta)} \R^n$). Note that the two summands in the definition of $F$ above are $p_K$-orthogonal by Lemma \ref{lem:PK}. As usual, it is convenient to identify each $T_{X_K(\theta)} \R^n$ with $\R^n$ and treat $F$ as a map from $\S^{n-1}$ to $\R^n$. 

Since $K$ is origin-symmetric and $f$ is even, $F$ is odd. Fixing an orthonormal basis $\{E_k\}_{i=1}^{n}$ of $(\R^n,\delta)$, now set
\[
\quad F_k := \scalar{ F,E_k} , \quad k=1,\ldots, n .
\]
Since $F_k$ is odd, $\int F_k dV_K = 0$; hence, we may apply the local Brunn-Minkowski inequality \eqref{L1BM} to each $F_k$ and sum over $k$, yielding
\begin{align}\label{eq:local-BM-F}
\sum_{k=1}^{n} \int |\grad F_k|_{g_K}^2 dV_K \geq (n-1) \sum_{k=1}^{n}\int F_k^2 dV_K .
\end{align}
The right-hand-side of (\ref{eq:local-BM-F}) is easy to evaluate using (\ref{eq:ab}) and Lemma \ref{lem:PK}:
\[
\sum_{k=1}^{n} F_k^2=|F|_{\delta}^2 \geq \beta^{-1} |F|_{p_K}^2=\beta^{-1}(|\grad f|_{g_K}^2+\tau^2 f^2).
\]

We now turn to the left-hand side of (\ref{eq:local-BM-F}).
Given $\theta_0\in \mathbb{S}^n$,  let $\{e_{i}\}_{i=1}^{n-1}$ be a local $g_K$-orthonormal frame for $T \S^{n-1}$ around $\theta_0$ which diagonalizes $\Theta$ at $\theta_0$, where $\Theta$ is the symmetric $(0,2)$-tensor given by
\[
\Theta :=\operatorname{Hess}^{\ast}_K f+\tau f g_K . 
\]
We use $\Theta^{\sharp}$ to denote the $(1,1)$ version of $\Theta$ after raising indices using the metric $g_K$, and abbreviate $\nabla = \nabla_K$. Let $\lambda_i$ denote the eigenvalue of $\Theta$ at $\theta_0$ corresponding to $e_i$.
In view of \eqref{centro-affine}, (\ref{eq:conjugate}) and \eqref{hess ast def}, we have 
\begin{align*}
	& dF_k(e_i) =\sscalar{ D_{e_i} F,E_k} \\
	&=\sscalar{ dX_K(\nabla_{e_i}\grad_{g_K} f)-g_K(e_i,\grad_{g_K} f)X_K + \tau f dX_K(e_i)+\tau df(e_i)X_K,E_k} \\
	&=\sscalar{ dX_K(\nabla_{e_i} \grad_{g_K} f)+ \tau f dX_K(e_i)+(\tau-1)df(e_i)X_K,E_k } \\
	&=\sscalar{ dX_K((\operatorname{Hess}^{\ast}_K f)^{\sharp}(e_i))+ \tau f dX_K (e_i)+(\tau-1)df(e_i)X_K ,E_k} \\
	& = \sscalar{dX_K( \Theta^{\sharp}(e_i)) + (\tau-1)df(e_i)X_K , E_k} . 
\end{align*}	
Therefore, at $\theta_0$ we have $dF_k(e_i)=\langle \lambda_{i} dX_K(e_i)+(\tau-1)df(e_i)X_K,E_k\rangle$ and
\begin{align*}
|\operatorname{grad}F_k|_{g_K}^2= \sum_{i=1}^{n-1} dF_k(e_i)^2 = \sum_{i=1}^{n-1} \sscalar{ \lambda_{i}dX_K(e_i)+(\tau-1)df(e_i)X_K,E_k}^2.
\end{align*}
Now from Lemma \ref{lem:PK}, we find that at $\theta_0$:
\begin{align*}
\sum_{k=1}^{n}|\operatorname{grad}F_k|_{g_K}^2&=\sum_{i=1}^{n-1}\sum_{k=1}^{n}\langle \lambda_{i}dX_K(e_i)+(\tau-1)df(e_i)X_K,E_k\rangle^2\\
&=\sum_{i=1}^{n-1} |\lambda_{i}dX_K(e_i)+(\tau-1)df(e_i)X_K|_{\delta}^2\\
&\leq \alpha^{-1}\sum_{i=1}^{n-1}|\lambda_{i}dX_K(e_i)+(\tau-1)df(e_i)X_K|_{p_K}^2\\
&= \alpha^{-1}\sum_{i=1}^{n-1} (\lambda_{i}^2+(\tau-1)^2 df(e_i)^2)\\
&= \alpha^{-1}(|\Theta|_{g_K}^2+(\tau-1)^2|\grad f|_{g_K}^2).
\end{align*}
Since both sides of the inequality do not depend on the local frame, this holds for all $\theta \in \S^{n-1}$. 

Plugging this into (\ref{eq:local-BM-F}), denoting $\sigma = \frac{\alpha}{\beta} = \frac{1}{\gamma} \in (0,1)$, applying (\ref{eq:Laplacian}) and integrating by parts, we obtain
\begin{align*}
& (n-1) \sigma \int (|\grad f|_{g_K}^2+\tau^2 f^2) dV_K\\
& \leq \int (|\Theta|_{g_K}^2+(\tau-1)^2|\grad  f|_{g_K}^2) dV_K \\
& = \int (|\operatorname{Hess}^{\ast}_K f|_{g_K}^2+2 \tau f \Delta_K f +(n-1)\tau^2 f^2 +(\tau-1)^2|\grad  f|_{g_K}^2) dV_K \\
& = \int (|\operatorname{Hess}^{\ast}_K f|_{g_K}^2+((\tau-1)^2-2\tau)|\grad f|_{g_K}^2+(n-1)\tau^2 f^2) dV_K . 
\end{align*}
Applying the centro-affine Bochner formula (\ref{eq:Bochner}), we find
\begin{align*}
	\int & \left [  (\Delta_K f)^2- (n+1+(n-1)\sigma -(\tau-2)^2)|\grad f|_{g_K}^2 \right .   \\
	                   & \left . +(n-1) \tau^2(1-\sigma)f^2 \right ] dV_K \geq 0.
\end{align*}
Integrating by parts again, we conclude that for all $C^2$-smooth even functions $f$ and choice of $\tau \in \R$:
\[
\int P_{\sigma,\tau}(-\Delta_K)(f) f dV_K \geq 0 ,
\]
where $P_{\sigma,\tau}$ is the following quadratic polynomial:
\[
 P_{\sigma,\tau}(t) := t^2-(n+1+(n-1)\sigma -(\tau-2)^2) t+ (n-1)\tau^2 (1-\sigma).
\] 

Consequently, the spectral theorem implies
\[
P_{\sigma,\tau}(\lambda_{1,e}(-\Delta_K)) \geq 0 \;\;\; \forall \tau \in \R .
\] 
It is straightforward to check that we have $P_{\sigma,\tau}(n-1) < 0$ for all $\tau-1 \in [0,\sqrt{ n \sigma/2}]$. Since $\lambda_{1,e}(-\Delta_K) > \lambda_1(-\Delta_K) = n-1$, it follows that $\lambda_{1,e}(-\Delta_K) \geq \sup_{\tau -1 \in [0,\sqrt{ n \sigma/2}]} R_{\sigma,\tau}$, where $R_{\sigma,\tau}$ denotes the right root of $P_{\sigma,\tau}$. 

To obtain a more elegant and useful bound, we make use of $\tau = 1 + \sigma$ (note that $\sigma \in [0,\sqrt{ n \sigma/2}]$), in which case a calculation verifies
\[
P_{\sigma,1+\sigma}(n-1 +\sigma (n+1)) = 2\sigma(\sigma^2-1) < 0 ,
\]
since $\sigma \in (0,1)$. 
We thus deduce that
\[
\lambda_{1,e}(-\Delta_K) > n-1 + \sigma(n+1) = n - p_\gamma,
\]
thereby concluding the proof of Theorem \ref{thm:main}.

When $\gamma=1$, we see that $P_{1,2}(2n) = 0$, and the above argument yields $\lambda_{1,e}(-\Delta_K) \geq 2n$. 
Remark \ref{rem:upper-bound} already implies that in that case $K$ must be a centered ellipsoid, but this is, of course, elementary: $\gamma=1$ implies that $P_{\ell K} = \alpha \delta$, and hence $\norm{x}^2_{\ell K} = \alpha |x|^2 + \scalar{x,v}$ for some $v \in \R^n$; but since $K$ is origin-symmetric, we must have $v = 0$ and so $\ell K$ is a Euclidean ball. In that case $\Delta_{\ell K}$ coincides with the usual Laplace-Beltrami operator $\Delta$ on $\S^{n-1}$, for which we have $\lambda_{1,e}(-\Delta) = 2n$ (see e.g. \cite{EMilman-Isospectral-HBM}), thereby concluding the proof of Theorem \ref{thm:spectral}. 
\end{proof}

\begin{remark}
The section $F : \S^{n-1} \rightarrow T_{\partial K} \R^n$, which we used in the proof, may be interpreted as the anisotropic gradient of $f : \S^{n-1} \rightarrow \R$ on $\partial K$ after extending it as a $\tau$-homogeneous function to $\R^n$. Namely, pushing forward $f$ onto $\partial K$ via $X_K$ and denoting by $\hat f$ its $\tau$-homogeneous extension to $\R^n$, one may check that $\hat f$'s gradient on $\partial K$ with respect to the anisotropic Riemannian metric $P_K$ is precisely $F$:
\[
\operatorname{grad}_{P_K} \hat f = dX_K(\operatorname{grad}_{g_K} f) + \tau f X_K = F .  
\]
It is therefore not surprising that we need to use the homogeneity parameter $\tau = 2$ for obtaining the sharp $p_{\gamma} = -n$ when $\gamma=1$ (corresponding to the case that $K$ is a Euclidean ball and $\Delta_K = \Delta$), as the even eigenfunctions of $\Delta$ corresponding to the eigenvalue $\lambda_{1,e}(-\Delta) = 2 n$ are quadratic polynomials.
\end{remark}

\def\cprime{$'$} \def\textasciitilde{$\sim$}

\end{document}